\documentclass[11pt,a4paper]{article} 

\usepackage{amsmath}			
\usepackage{amsfonts}			
\usepackage{amssymb}			

\newtheorem{theorem}{Theorem}
\newtheorem{proposition}[theorem]{Proposition}
\newtheorem{corollary}[theorem]{Corollary}

\newtheorem{lemma}[theorem]{Lemma}
\newtheorem{remark}[theorem]{Remarks}
\newenvironment{proof}{\noindent{\bf Proof.}\ }
{$\bullet$\medskip\par}

\def\cX{{\mathcal X}}
\def\Ha{\textrm{H}}
\def\({\left(}
\def\){\right)}
\def\<{\langle}
\def\>{\rangle}

\title{\bf  Second moment of Dirichlet $L$-functions, \\
character sums over subgroups, \\ and upper bounds on relative class numbers.}
\author{
St\'ephane R. LOUBOUTIN\\
Aix Marseille Universit\'e, CNRS, Centrale Marseille, I2M,\\ 
Marseille, FRANCE\\
stephane.louboutin@univ-amu.fr\\
Marc MUNSCH\\
5010 Institut f\"{u}r Analysis und Zahlentheorie\\
8010 Graz, Steyrergasse 30, Graz, AUSTRIA\\
munsch@math.tugraz.at}
\date{\today}

\begin{document}
\bibliographystyle{alpha}
\maketitle

\footnotetext{
2020 Mathematics Subject Classification. 
Primary. 11R42, 11L40. Secondary: 11A07, 11J71, 11M20, 11R18, 11R20.
Key words and phrases: moments of Dirichlet $L$-functions, 
 cyclotomic field, 
relative class number, character sums, multiplicative subgroups,
exponential sums,
Dedekind sums, 
discrepancy.}

\begin{abstract}
We prove an asymptotic formula for the mean-square average of $L$- functions associated to subgroups of characters of sufficiently large size. Our proof relies on the study of certain character sums ${\cal A}(p,d)$ recently introduced by E. Elma. We obtain an asymptotic formula for ${\cal A}(p,d)$ which holds true for any divisor $d$ of $p-1$ removing previous restrictions on the size of $d$. This anwers a question raised in Elma's paper. Our proof relies both on estimates on the frequency of large character sums and techniques from the theory of uniform distribution. As an application we deduce the following bound $h_{p,d}^-
\leq 2\left (\frac{(1+o(1))p}{24}\right )^{m/4}$ on the relative class numbers 
of the imaginary number fields of conductor $p\equiv 1\mod d$ and degree $m=(p-1)/d$. 

\end{abstract}

\section{\bf Introduction}

Throughout the paper $d\geq 1$ will be an odd integer 
and $p$ will be an odd prime satisfying $p\equiv 1\pmod {2d}$. We also write $\log_j$ for the j-th iterated logarithm. We let $\cX_p$ denote the multiplicative cyclic group of order $p-1$ of the Dirichlet characters mod $p$ 
and let $\cX_p^*$ denote the set with $p-2$ elements of the non-trivial Dirichlet characters mod $p$.
 We set $m =(p-1)/d$, an even integer, 
and $\chi_{p,m}$ will denote any one of the $\phi (m)$ odd Dirichlet characters in $\cX_p$ of order $m$. \\

Let $h_{p,d}^-$ be the relative class number of 
the imaginary subfield $K_{p,d}$ of the cyclotomic field ${\mathbb Q}(\zeta_p)$ 
of even degree $(K_{p,d}:{\mathbb Q})=m$ 
and odd relative degree $({\mathbb Q}(\zeta_p):K_{p,d}) =d$ 
(e.g. see \cite[Chapter 4]{Was}). For $d=1$, we have $K_{p,1} ={\mathbb Q}(\zeta_p)$ 
and it has long been known that 
\begin{equation}\label{bounp1}
h_{p,1}^-
=h_{{\mathbb Q}(\zeta_p)}^-
\leq 2p\left (\frac{p}{24}\right )^{(p-1)/4}
=2p\left (\frac{p}{24}\right )^{m/4},
\end{equation}
see \cite{Met} and \cite{Wal}. 
In \cite{{LouCMB36/37}} it is explained how to improve upon this bound 
by taking values greater than $24$ for this denominator, in fact values as close to $4\pi^2$ as desired. 
See also \cite{Gra} for more subtle results.\\

Denote by $w_{p,d}$ the number of complex roots of unity contained in $K_{p,d}$, we have $w_{p,1} =2p$ and $w_{p,d}=2$ for $d>1$. The following bound holds:
\begin{equation}\label{formulahrel}
h_{p,d}^-
=w_{p,d}\prod_{j=1}^{m/2}
\frac{\sqrt p}{2\pi}L(1,\chi_{p,m}^{2j-1})
\leq w_{p,d}\left (\frac{pM(p,m)}{4\pi^2}\right )^{m/4},
\end{equation}
where $M(p,m)$ denotes the following mean square of $L(1,\chi)$:
\begin{equation}\label{defMpm}
M(p,m)
:=\frac{2}{m}\sum_{j=1}^{m/2}\vert L(1,\chi_{p,m}^{2j-1})\vert^2.
\end{equation} Therefore explicit formulas (or asymptotic formulas) for $M(p,m)$ allow to give precise upper bounds of type $h_{p,d}^-\leq  C_1\cdot C_2^{m/4}$. For $d=1$, H. Walum deduced (\ref{bounp1}) in \cite{Wal} by proving that 
\begin{equation}\label{Mp1}
M(p,p-1) 
=\frac{\pi^2}{6}\left (1-\frac{1}{p}\right )\left (1-\frac{2}{p}\right ).
\end{equation}  For $d=3,5$, some explicit formulas for $M(p,m)$ have been obtained in certain cases by the first author (see Section \ref{explicit}) allowing him to give upper bounds on $h_{p,3}^-$ and $h_{p,5}^-$. In contrast, for a given even $m$, 
as $p$ runs over the prime integers $p\equiv 1+m\mod{2m}$
nothing better than $h_{p,(p-1)/m}^- \ll_m (p\log^2 p)^{m/4}$ is known, 
by using (\ref{formulahrel}) and the bound $\vert L(1,\chi)\vert\ll\log p$. To begin with, the following simple argument gives a trivial bound on $M(p,m)$. Since in $M(p,m)$ we consider only $m/2$ of the $(p-1)/2$ odd Dirichlet characters that appear in $M(p,p-1)$,  
we have 
\begin{eqnarray}\label{trivialMp}
M(p,m)
&\leq\frac{2}{m}\frac{p-1}{2}M(p,p-1)
=dM(p,p-1)
=d\frac{\pi^2}{6}\left (1-\frac{1}{p}\right )\left (1-\frac{2}{p}\right ).
\end{eqnarray} By (\ref{formulahrel}) it implies   
\begin{equation}\label{trivial} h_{p,d}^-
\leq 2\left (\frac{dp}{24}\right )^{\frac{p-1}{4d}}
=2\left (\frac{dp}{24}\right )^{m/4}.\end{equation} 
 The aim of the paper is to give an asymptotic formula for $M(p,m)$ when $m$ is of reasonable size with respect to $p$ and an upper bound when $m$ is small. As a consequence we obtain a significant improvement upon the trivial bound \eqref{trivial}.

\noindent\frame{\vbox{
\begin{theorem}\label{mainth1}
As $p$ tends to infinity and $d\geq 1$ runs over the odd divisors of $p-1$ 
such that $d \leq (1/2-\varepsilon)\frac{\log p}{\log_2 p}$, we have the asymptotic formula
\begin{equation}\label{asympmsq}
M(p,m)=M(p,(p-1)/d)
=\frac{\pi^2}{6}\left(1+ O(d(\log p)^2 p^{-\frac{1}{d-1}})\right),
\end{equation}
which implies the upper bound 
\begin{equation}\label{nonexplicit}
h_{p,d}^-
\leq 2\left (\frac{(1+o(1))p}{24}\right )^{m/4}.
\end{equation} If the previous bound does not apply but $d$ is such that $\log d= o(\log p/\log_2 p)$, we have for some absolute constant $C> 0$
\begin{equation}\label{boundMpmlarge}
M(p,m)=M(p,(p-1)/d) \leq C (\log_2 d)^2
\end{equation}
which implies, in this range of $d$, the upper bound
\begin{equation}\label{boundhplarge}
h_{p,d}^-
\leq (Cp(\log_2 d)^2)^{m/4}.
\end{equation} 

\end{theorem}
}}

\begin{remark} 
 It should be emphasized that the error term in \eqref{asympmsq} is almost optimal, in view of Proposition \ref{propdedekind}. Furthermore \eqref{nonexplicit}  is in accordance with the known asymptotics (see \cite[Theorem 4]{LouManuMath91})
$\log h_{p,d}^-\sim\frac{m+o(1)}{4}\log p.$ For very large $d$, the bound $M(p,m) \ll \log^2 p$ remains the best known. This is not surprising if we look at the very extreme case $d=(p-1)/2$ and $p\equiv 3\pmod 4$. 
In that situation, $\chi$ is the quadratic character given by the Legendre symbol $\chi (n) =\left (\frac{n}{p}\right )$ and $M(p,m)=\vert L(1,\chi)\vert^2$. Under GRH, Littlewood \cite{Littlewood} proved that $L(1,\chi) \ll \log_2 p$  but improving upon the bound $L(1,\chi) \ll \log p$ remains out of reach unconditionally. On the other way we cannot expect an uniform bound for $M(p,m)$ better than $(\log_2 p)^2$. Indeed, Chowla \cite{Chowla} proved unconditionally that there are infinitely many quadratic characters $\chi$ such that $L(1,\chi) \gg \log_2 p$. This supports the hypothesis that the bound \eqref{boundMpmlarge} could be sharp.
\end{remark}

The paper is organized as follows. To begin, in Section \ref{Elmasection}, 
we recall (and give a simple proof of) a formula discovered by Elma 
relating $M(p,m)$ to certain character sums ${\cal A}(p,d)$ defined below in \eqref{defApchi}.
In Section \ref{primespecial}, Proposition \ref{propdedekind}, 
we show that for a certain family of primes $p$ and $d$ 
we can compute exactly $M(p,m)$ using properties of Dedekind sums.
Finally, in Section \ref{sectionasymptotic}, we prove an asymptotic formula for ${\cal A}(p,d)$ which directly implies Theorem \ref{mainth1} (see Section $5$). A crucial point of the analysis comes from the fact that the average in \eqref{defMpm} is made over a family of $m/2$ characters which could be of any size with respect to $p$. The same difficulty carries into the analysis of  ${\cal A}(p,d)$ with a character sum averaged over a subgroup of $\mathbb{F}_{p}^*$ of size $d$. On one hand when $d$ is small (see Theorem \ref{asympdpetit}) we write ${\cal A}(p,d)$ as an average of a function evaluated at equidistributed points  modulo $1$ and use techniques from discrepancy theory. On the other hand, when $d$ is large (see Theorem \ref{asymptoticlarge}) we rely on character sums techniques and incorporate recent estimates on the frequency of large character sums \cite{bober2018frequency}.

\section{Elma's character sums}\label{Elmasection}
Let $\chi$ be an odd Dirichlet character of (even) order $m$ dividing $p-1$ and prime conductor $p\geq 3$. 
Set $d =(p-1)/m$ (an odd integer) and 
\begin{equation}\label{defApchi}
{\cal A}(p,d)
=\frac{1}{p-1}
\sum_{N=1}^{p-1}\left (\sum_{1 \leq n_1,n_2 \leq N \atop \chi(n_1)=\chi(n_2)} 1\right )
\end{equation}
(the results depends only on $p$ and $d$, not on the choice of $\chi$).

\subsection{Link with the mean square value $M(p,m)$}

E. Elma proved a nice connection between the mean square values $M(p,m)$'s and these character sums 
${\cal A}(p,d)$.
We give a simple and short proof of \cite[Theorem 1.1]{Elma}:

\noindent\frame{\vbox{
\begin{theorem}\label{thElma}
Let $\chi$ be a primitive Dirichlet modulo $f>2$, its conductor. 
Set $S(k,\chi) =\sum_{l=0}^k\chi (l)$ 
and let $L(s,\chi)=\sum_{n\geq 1}\chi (n)n^{-s}$ be its associated Dirichlet $L$-series. 
Then 
$$\sum_{k=1}^{f-1}\vert S(k,\chi)\vert^2 
=\frac{f^2}{12}\prod_{p\mid f}\left (1-\frac{1}{p^2}\right )
+a_\chi\frac{f^2}{\pi^2}\vert L(1,\chi)\vert^2, 
\hbox { where }
a_\chi
:=\begin{cases}
0&\hbox{if $\chi (-1)=+1$,}\\
1&\hbox{if $\chi (-1)=-1$.}
\end{cases}$$
\end{theorem}
}}

\begin{proof}
Our simple proof is based on an easy to remember idea: 
we apply Parseval's formula 
$\int_0^1\vert F(x)\vert^2{\rm d}x
=\sum_{n=-\infty}^\infty\vert c_n(F)\vert^2$ 
to the function $x\in [0,1)\mapsto F(x):=\sum_{0\leq l\leq fx}\chi (l)$ extended to $x\in {\mathbb R}$ by $1$-periodicity.
The reader would be able to reconstruct the argument using this simple idea.
Let us now give all the details.
Since $\chi$ is primitive,
the Gauss sums 
$\tau(n,\chi)
=\sum_{k=1}^f\chi (k)\exp(2\pi ink/f)$ 
and $\tau(\chi) =\tau(1,\chi)$ satisfy $\tau (n,\chi) =\overline{\chi (n)}\tau(\chi)$ and $\vert\tau (\chi)\vert ^2=f$, 
e.g. see \cite[Lemmas 4.7 and 4.8]{Was}. 
(These properties are easy to check when $f=p\geq 3$ is prime).
Since $x\mapsto F(x)=S(k,\chi)$ is constant for $x\in [k/f,(k+1)/f)$, 
we have
$$\int_0^1\vert F(x)\vert^2{\rm d}x
=\frac{1}{f}\sum_{k=0}^{f-1}\vert S(k,\chi)\vert^2.$$
and the Fourier coefficients 
of $F$ are given by 
\begin{equation}\label{cnF}
c_n(F)
=\int_0^1F(x)\exp(-2\pi i nx){\rm d}x\\
=\sum_{k=0}^{f-1}S(k,\chi)\int_{k/f}^{(k+1)/f}\exp(-2\pi i nx){\rm d}x.
\end{equation}
Hence, by \cite[Theorem 4.2]{Was} we have 
$$c_0(f)
=\frac{1}{f}\sum_{k=0}^{f-1}S(k,\chi)
=\frac{1}{f}\sum_{k=0}^{f-1}\sum_{l=0}^k\chi (l)
=\frac{1}{f}\sum_{l=0}^{f-1}(f-l)\chi (l)
=-\frac{1}{f}\sum_{l=0}^{f-1}l\chi (l)
=L(0,\chi)$$
and for $n\neq 0$ we have 
\begin{eqnarray*}
c_n(F)
&=&\sum_{k=0}^{f-1}S(k,\chi)\frac{\exp\left (-\frac{2\pi in(k+1)}{f}\right )-\exp\left (-\frac{2\pi ink}{f}\right )}{-2\pi in}
\ \ \ \ \ \hbox{(by (\ref{cnF}))}\\
&=&\sum_{k=1}^{f-1}\frac{(S(k,\chi)-S(k-1,\chi))\exp\left (-\frac{2\pi ink}{f}\right )}{2\pi in}
\ \ \ \ \ \hbox{(notice that $S(0,\chi)=S(f-1,\chi)=0$)}\\
&=&\frac{\tau (-n,\chi)}{2\pi i n}
=\frac{\tau(\chi)}{2\pi i}\times\frac{\overline{\chi (-n)}}{n}
=-\chi (-1)c_{-n}(F).
\end{eqnarray*}
Now, $L(0,\chi)=0$ if $\chi (-1)=+1$ and $\vert L(0,\chi)\vert^2 =\frac{f}{\pi^2}\vert L(1,\chi)\vert^2$ if $\chi (-1)=-1$, 
e.g. see \cite[Chapter 4, page 30]{Was}. 
Therefore, Parseval's formula gives
$$\frac{1}{f}\sum_{k=0}^{f-1}\vert S(k,\chi)\vert^2
=a_\chi \frac{f}{\pi^2}\vert L(1,\chi)\vert^2
+2\sum_{n\geq 1\atop\gcd (n,f)=1}\frac{f}{4\pi^2}\times\frac{1}{n^2}
=a_\chi\frac{f}{\pi^2}\vert L(1,\chi)\vert^2
+\frac{f}{12}\prod_{p\mid f}\left (1-\frac{1}{p^2}\right )$$
and the desired result follows. 
Notice that this proof is similar to the ones in \cite{BC}.
\end{proof}

\noindent\frame{\vbox{
\begin{corollary}\label{corElma}
Let $\chi$ be an odd Dirichlet character of (even) order $m$ dividing $p-1$ and prime conductor $p\geq 3$. 
Set $d =(p-1)/m$ (an odd integer) and 
let ${\cal A}(p,d)$ be as in (\ref{defApchi}). 
Then 
\begin{equation}\label{MA}
M(p,m)
:=\frac{2}{m}\sum_{j=0\atop j\ odd}^{m-1}\vert L(1,\chi^{j})\vert^2
=\frac{\pi^2}{6}\frac{p-1}{p^2}
\left (12{\cal A}(p,d)
-(4d+1)p-d-1\right ).
\end{equation}
In particular, 
\begin{equation}\label{boundingApd}
\frac{(4d+1)p+d+1}{12}
\leq {\cal A}(p,d)
\leq\frac{(5d+1)p+d+1}{12}.
\end{equation} 
Moreover, by \cite{LouCRAS323}, we have
\begin{equation}\label{boundMpm}
0
\leq M(p,m)
\leq (\log p +2+\gamma-\log\pi)^2/4.
\end{equation}
\end{corollary}
}}

\begin{proof}
By Theorem \ref{thElma}, for $j$ odd we have 
$$\vert L(1,\chi^j)\vert^2
=-\frac{\pi^2}{12}\left (1-\frac{1}{p^2}\right )
+\frac{\pi^2}{p^2}\sum_{k=1}^{p-1}\vert S(k,\chi^j)\vert^2$$
The $\chi^j$'s are primitive modulo $p$ for $1\leq j\leq m-1$, 
whereas $\chi^0$ is the non-primitive trivial Dirichlet character modulo $p$. 
Therefore, on the one hand we have
$$\sum_{j=0}^{m-1}\sum_{k=1}^{p-1}\vert S(k,\chi^{j})\vert^2
=\sum_{j=0}^{m-1}\sum_{k=1}^{p-1}\left\vert\sum_{l=1}^k\chi^{j} (l)\right\vert^2
=\sum_{j=0}^{m-1}\sum_{k=1}^{p-1}\sum_{1\leq l_1,l_2\leq k}\chi^{j}(l_1)\overline{\chi^{j}(l_2)}
=m(p-1){\cal A}(p,d),$$
by using the orthogonality relation
$$\sum_{j=0}^{m-1}\chi^j (n_1)\overline{\chi^j (n_2)}
=\begin{cases}
m&\hbox{if $\chi (n_1)=\chi(n_2)\neq 0$,}\\
0&\hbox{otherwise.}\\
\end{cases}$$
On the other hand, Theorem \ref{thElma} gives 
$$\sum_{j=1}^{m-1}\sum_{k=1}^{p-1}\vert S(k,\chi^{j})\vert^2
=\frac{(m-1)(p^2-1)}{12}
+\frac{mp^2}{2\pi^2}M(p,m).$$
Since
$$\sum_{k=1}^{p-1}\vert S(k,\chi^{0})\vert^2
=\sum_{k=1}^{p-1}\left\vert\sum_{l=1}^k 1\right\vert^2
=\sum_{k=1}^{p-1} k^2 
=\frac{(p-1)p(2p-1)}{6},$$
it follows that 
$$m(p-1){\cal A}(p,d)
=\frac{m(p^2-1)}{12}
+\frac{(p-1)^2(4p+1)}{12}
+\frac{mp^2}{2\pi^2}M(p,m).$$
The desired identity (\ref{MA}) follows.

Now, noticing that $M(p,m)\geq 0$, the lower bound on ${\cal A}(p,d)$ in (\ref{boundingApd}) follows from (\ref{MA}).
Finally, by \eqref{trivialMp} we have
$$M(p,m)
\leq\frac{d\pi^2}{6}\left (1-\frac{1}{p}\right ).$$
Plugging this bound in (\ref{MA}) we obtain the upper bound on ${\cal A}(p,d)$ in (\ref{boundingApd}). 
\end{proof}

By (\ref{formulahrel}) and (\ref{MA}), upper bounds on ${\cal A}(p,d)$ 
would yield upper bounds on $h_{p,d}^-$.
More precisely, 
for $d>1$,
$M(p,m)\leq\pi^2/6$ which is equivalent to ${\cal A}(p,d)<\bigl ((4d+2)p+d+2\bigr )/12$ 
would yield $h_{p,d}^-\leq 2(p/24)^{m/4}$.

\begin{remark}
Corollary \ref{corElma} gives $\frac{13p+4}{12}\leq {\cal A}(p,3)\leq\frac{16p+4}{12}$,
whereas ${\cal A}(p,3) =\frac{14p+4}{12}$, 
by (\ref{Ap3}) below. 
Hence it should be possible to improve upon the upper bound in \eqref{boundingApd}.
\end{remark}

\subsection{Exact formulas for $M(p,m)$ and ${\cal A}(p,d)$ in specific cases}\label{explicit}
There are only four cases listed below where an explicit formula for ${\cal A}(p,d)$ is known.

\begin{enumerate}

\item By (\ref{Mp1}) and (\ref{MA}), for $d=1$ we have 
\begin{equation}\label{Ap1}
{\cal A}(p,1)
=\frac{p}{2}
=\frac{(2d+1)p}{6}.
\end{equation}

\item For $d=3$ we proved in \cite[Theorem 1]{LouBPASM64} 
that 
\begin{equation}\label{Mp3}
M(p,(p-1)/3) =\frac{\pi^2}{6}\left (1-\frac{1}{p}\right )
\ \ \ \ \ \hbox{(for $p\equiv 1\pmod 6$)}
\end{equation} 
 and the corresponding bound on the relative class number
\begin{equation}\label{bounp3}
h_{p,3}^-
\leq 2\left (\frac{p}{24}\right )^{(p-1)/12}
=2\left (\frac{p}{24}\right )^{m/4}.
\end{equation}

By (\ref{MA}), 
this gives for $d=3$ and $p\equiv 1\pmod 6$, 
\begin{equation}\label{Ap3}
{\cal A}(p,3) 
=\frac{7p+2}{6}
=\frac{(2d+1)p}{6}+o(p).
\end{equation}

\item For $d=5$ we proved in \cite[Theorem 5]{LouBPASM64} that

\begin{equation}\label{Mp5}
M(p,(p-1)/5) 
=\frac{\pi^2}{6}\left (1+\frac{2a(a+1)^2-1}{p}\right )
\ \ \ \ \ \hbox{(for $p>5$ of the form $p =\frac{a^5-1}{a-1}$)}.
\end{equation} and the corresponding bound on the relative class number
\begin{equation}\label{bounp5}
h_{p,5}^-
\leq 2\left (\frac{p}{24}\right )^{(p-1)/20}
=2\left (\frac{p}{24}\right )^{m/4}.
\end{equation}
By (\ref{Mp5}) and (\ref{MA}), this implies 

\begin{equation}\label{Ap5}
{\cal A}(p,5)
=\frac{11p+3}{6}+\frac{a(a+1)^2p}{6(p-1)}
=\frac{(2d+1)p}{6}+o(p).
\end{equation}

\item For $d=(p-1)/2$ and $3<p\equiv 3\pmod 4$. 
In that situation, $\chi$ is the quadratic character given by the Legendre symbol $\chi (n) =\left (\frac{n}{p}\right )$, 
$L(1,\chi) =\pi h_{{\mathbb Q}(\sqrt{-p})}/\sqrt p$ and (\ref{MA}) gives 
$${\cal A}(p,(p-1)/2)
=\frac{4p^2-p+1}{24}+\frac{ph_{{\mathbb Q}(\sqrt{-p})}^2}{2(p-1)}.$$
\end{enumerate}
\begin{remark}
In fact, $d=1$ is the only case for which we could come up with a direct proof of the formula for ${\cal A}(p,d)$.
Indeed, we have $\chi_{p,p-1}(n_1)=\chi_{p,p-1}(n_2)$ if and only if $n_1\equiv n_2\pmod p$. 
Hence,
$${\cal A}(p,1)
=\frac{1}{p-1}\sum_{N=1}^{p-1} N 
=\frac{p}{2}.$$
It would be nice to have similar independent and direct proofs of (\ref{Ap3}) and (\ref{Ap5}). 
\end{remark}

\section{Evaluation of $M(p,m)$ for primes $p=(a^d-1)/(a-1)\equiv 1\pmod {2d}$}\label{primespecial}
We gave an explicit formula for ${\cal A}(p,3)$, see (\ref{Ap3}), 
and one for ${\cal A}(p,5)$, but only for the primes $p$ of the form $p=(a^5-1)/(a-1)$, see (\ref{Ap5}).
After some numerical computation for primes of the form $(a^5-2^5)/(a-2)$ or $(a^5-3^5)/(a-3)$ 
we could not guess any formula for $M(p,(p-1)/5)$ or ${\cal A}(p,5)$.
However, we now prove a general result which recover \eqref{Mp3} and \eqref{Mp5}(let us say that we forgot to deal with the case $a<0$ in the proof of \cite[Theorem 5]{LouBPASM64}). 
We want to point out that here again we do not directly compute ${\cal A}(p,d)$.
Instead we give an exact formula for $M(p,m)$ and then use (\ref{MA}) to deduce an expression for ${\cal A}(p,d)$.

\noindent\frame{\vbox{
\begin{proposition}\label{propdedekind}
Set 
$Q_l(X) 
=(X^l-1-l(X-1))/(X-1)^2\in {\mathbb Z}[X]$, $l\geq 1$. 
Hence, $Q_1(X) =0$, $Q_2(X)=1$
and $Q_l(X)=X^{l-2}+2X^{l-3}+\cdots+(l-2)X+(l-1)$ for $l\geq 2$.
Let $d\geq 3$ be a prime integer. 
For a prime integer of the form $p=(a^d-1)/(a-1) $ for some $a\neq -1,0,1$, 
we have 
\begin{equation}\label{Mpd}
M(p,(p-1)/d)
=\frac{\pi^2}{6}\left (1+\frac{2a(a+1)^2Q_{\frac{d-1}{2}}(a^2)-1}{p}\right ),
\end{equation}
\begin{equation}\label{Apd}
{\cal A}(p,d) 
=\frac{(2d+1)p+\frac{d+1}{2}}{6}
+\frac{p}{p-1}\cdot\frac{a(a+1)^2Q_{\frac{d-1}{2}}(a^2)}{6}
=\frac{2d+1}{6}p+O\left (p^{1-\frac{1}{d-1}}\right ),
\end{equation}
and 
$$h_{p,d}^-\leq 2\left (p/24\right )^{m/4}
\hbox{ for $p=(a^d-1)/(a-1)$ with $a\leq -2$}.$$
\end{proposition}
}}

\begin{proof}
We keep the notation of \cite{LouBPASM64}, 
use the properties of Dedekind sums 
$$s(c,d)
={1\over 4d}\sum_{n=1}^{d-1}\cot \left ({\pi n\over d}\right )\cot \left ({\pi nc\over d}\right )
\ \ \ \ \ (c\in {\mathbb Z},\ d\in {\mathbb Z}\setminus\{-1,0,1\})$$
recalled in \cite{LouBPASM64}
and set $l=(d-1)/2$. 
To deal in one stroke with the two cases $a\leq -2$ and $a\geq 2$ 
we have extended the definition of Dedekind sums, allowing $d$ to be negative. 
Letting $\epsilon(d)\in\{\pm 1\}$ denote the sign of $0\neq d\in {\mathbb Z}$, 
the reciprocity and complementary laws for these generalized Dedekind sums are 
$$s(c,d)+s(d,c) 
=\frac{c^2+d^2-3\epsilon(c)\epsilon(d)cd +1}{12cd}
\hbox{ and }
s(1,d) 
=\frac{d^2-3\epsilon(d)d+2}{12d}.$$
Set $p_k =(a^k-1)/(a-1)$ and $\epsilon=\epsilon(a)$. 
Then $\epsilon(p)=1$, $\epsilon(a^k) =\epsilon^k$ and $\epsilon (p_k) =\epsilon^{k+1}$. 
We have 
$$M(p,(p-1)/d)
=\frac{\pi^2}{6}\left (1+\frac{N}{p}\right ),
\hbox{ where }
N
=24\left (\sum_{k=1}^{l}
s(a^k,p)
\right )
-3
+\frac{2}{p}.$$
Now, $p\equiv p_k\pmod{a^k}$ and $a^k\equiv 1\pmod {p_k}$ for $1\leq k\leq d$. 
Hence 
$$s(a^k,p) 
=\frac{a^{2k}+p^2-3\epsilon^ka^kp+1}{12a^kp}
-s(p,a^k)
=\frac{a^{2k}+p^2-3\epsilon^ka^kp+1}{12a^kp}
-s(p_k,a^k)$$
and
$$s(p_k,a^k)
=\frac{p_k^2+a^{2k}-3\epsilon p_ka^k+1}{12p_ka^k}
-s(a^k,p_k)
=\frac{p_k^2+a^{2k}-3\epsilon p_ka^k+1}{12p_ka^k}
-s(1,p_k),$$
by the reciprocity law for Dedekind sums. 
Since
$$s(1,p_k)
=\frac{p_k^2-3\epsilon^{k+1}p_k+2}{12p_k},$$ 
by the complementary law for Dedekind sums,
we obtain
$$s(a^k,p_k)
=\frac{a^{2k}+p^2-3\epsilon^ka^kp+1}{12a^kp}
-\left (\frac{p_k^2+a^{2k}-3\epsilon p_ka^k+1}{12p_ka^k}
-\frac{p_k^2-3\epsilon^{k+1}p_k+2}{12p_k}
\right )$$
and
$$s(a^k,p_k)
=\frac{a^{2k}+p^2-3a^kp+1}{12a^kp}
+(a-1)\frac{p_k^2+1-a^k}{12a^k}.$$
Notice that the more natural congruence $p\equiv p_{k-1}\pmod{a^k}$ and $a^k\equiv a\pmod {p_{k-1}}$ 
would lead to slightly more complicated computations. 
An easy but boring computation using $\sum_{k=1}^l b^k =b(b^l-1)/(b-1)$
then finally yields $N=2a(a+1)^2Q_{l}(a^2)-1$, as desired.
\end{proof}

\begin{remark}  It is widely believed since a long time that there are infinitely many primes of the form  $p=(a^d-1)/(a-1)$, as firstly investigated in the special case  of Mersenne primes $2^p-1$ ($a=2$). More precise results about the number of such primes less than $x$ are expected. This is sometimes called Lenstra-Pomerance-Wagstaff conjecture (see the survey \cite{Mersenne} for more information and references on this topic).  \end{remark}

\section{Asymptotic behavior of Elma's sums}\label{sectionasymptotic}
Let us remark that $$\frac{{\cal A}(p,d)}{dp/3}
=\frac{pM(p,m)}{2\pi^2d(p-1)}
+1+\frac{1}{4d}+\frac{1}{4p}+\frac{1}{4dp},$$
by (\ref{MA}). Hence,  as $d/\log^2p$ tend to infinity, we have by (\ref{boundMpm})
\begin{equation}\label{Elmaremark} {\cal A}(p,d) \sim \frac{dp}{3} \end{equation} as noticed in \cite{Elma}. As conjectured by Elma, we could expect the same behavior in a wider range of $d$. Our goal in this section is to prove that \eqref{Elmaremark}  holds true without any restriction on the size of the parameter $d$. Moreover when $d$ is constant, we obtain a refined asymptotic formula $A(p,d)\sim \frac{2d+1}{6}$  which is in accordance with the exact formulas from Section \ref{explicit}.

\noindent\frame{\vbox{
\begin{theorem}\label{conjecture}
As $p$ tends to infinity and $d\geq 1$ runs over the odd divisors of $p-1$, 
we have 
$${\cal A}(p,d) =\frac{dp}{3} +o(dp).$$
\end{theorem}
}}

We will split the discussion into two cases depending on whether $d$ goes or not to infinity.
Theorem \ref{conjecture} follows from Theorems \ref{asympdpetit} ($d$ small) and \ref{asymptoticlarge} ($d$ large) proved below. In the former case, we obtain the more precise asymptotic expansion $A(p,d)\sim \frac{2d+1}{6}$. By (\ref{MA}), this allows us to deduce an asymptotic formula for $M(p,m)$. In the latter case, Theorem \ref{conjecture} is not sufficient to infer an asymptotic formula for $M(p,m)$ and only implies an upper bound.

\subsection{Asymptotic for small $d$'s}
Our goal in this section is to prove the following theorem which gives Theorem \ref{conjecture} for small $d$'s:

\noindent\frame{\vbox{
\begin{theorem}\label{asympdpetit}
Let $d$ range over the odd integers. 
Set $\gamma(d)=\max_{k \mid d}\phi(k)$. 
Hence $\gamma(d)\leq d-1$, with equality whenever $d$ is an odd prime.
Let $p$ range over the prime integers such that $p\equiv 1 \pmod {2d}$. 
Then we have the following asymptotic formula
$${\cal A}(p,d)
= \frac{2d+1}{6}p + O\left(d(\log p)^{2}p^{1-1/\gamma(d)} \right)$$ 
where the implicit constant in the error term is absolute.\\ 
In particular if $d \leq c\frac{\log p}{\log_2 p}$ with $c<1/2$, we have 
$${\cal A}(p,d)\sim \frac{2d+1}{6}p.$$
\end{theorem}
}}

\begin{remark} 
Let us point out that (\ref{Apd}) shows that the power of $p$ in the error term of Theorem \ref{asympdpetit} is optimal. 
\end{remark}

\subsubsection{Results from uniform distribution theory}
For any fixed integer $s$, we consider the $s$-dimensional cube $I_s=\left[0,1\right]^s$ 
equipped with its $s$-dimensional Lebesgue measure $\lambda_s$. 
We denote by $\mathcal{B}$ the set of rectangular boxes of the form 
$$\prod_{i=1}^{s}[\alpha_i,\beta_i)
=\left\{x\in I_s, \alpha_i\leq x_i <\beta_i\right\}$$ where $0\leq \alpha_i<\beta_i\leq 1.$

If $S$ is a finite subset of $I^s$, we define the discrepancy $D(S)$ by 
$$D(S)
=\sup_{B \in \mathcal{B}}\left\vert \frac{\# (B\cap S)}{\# S}-\lambda_s(B)\right\vert.$$
The discrepancy measures in a quantitative way the deviation of a pointset $S$ from equidistribution. 
In particular a sequence of sets $S_n$ is uniformly distributed if and only if $D(S_n) \xrightarrow[n\to \infty]{} 0$. 
More precisely we have the \textbf{Koksma-Hlawka inequality}:

\begin{theorem}\cite[Theorem $1.14$]{tichy}\label{koksma}
Let $f(\mathbf{x})$ a function of bounded variation on $I_s$ in the sense of Hardy and Krause and 
$\mathbf{x_1},\dots,\mathbf{x_N}$ a finite sequence of points in $I_s$. 
Then
$$\left\vert \frac{1}{N} \sum_{i=1}^{N}f(\mathbf{x_i})-\int_{I_s} f(u)d\lambda_s(u) \right\vert 
\leq V(f)D(S)$$
 where $V(f)$ is the Hardy-Krause variation of $f$ (see also \cite[Chapter $2$]{nied}).
 \end{theorem} 
 
In order to estimate the discrepancy, we recall the inequality of \textbf{Erd\H{o}s-Tur\'{a}n-Koksma}: 

\begin{theorem}\cite[Theorem $1.21$]{tichy}.
Let $S=\left\{\mathbf{x_1},\dots,\mathbf{x_N}\right\}$ be a set of points in $I_s$ and $H$ a positive integer. 
Then we have 
\begin{equation}\label{erdosturan} 
D(S) \leq \left(\frac{3}{2}\right)^s
\left(
\frac{2}{H+1} 
\sum_{0<\left\|\mathbf{h}\right\|_{\infty}\leq H} \frac{1}{r(\mathbf{h})}
\left\vert \frac{1}{N}\sum_{n=1}^{N}e(\langle \mathbf{h},\mathbf{x_N}\rangle)\right\vert
\right),
\end{equation}
where 
$e(z) = \exp(2 \pi i z)$,
$\displaystyle{r(\mathbf{h})
=\prod_{i=1}^{s} \max\{1,\vert h_i\vert\}}$ 
for $\mathbf{h}=(h_1,\dots,h_s) \in \mathbb{Z}^s$ 
and $\langle , \rangle$ denotes the standard inner product in $\mathbb{R}^{s}$.
\end{theorem}

\subsubsection{Notions from pseudo random generators theory}\label{pseudo}
In the rest of the paper the reults of the previous section will only be used for $s=2$. \\ We introduce some tools from the theory of pseudo-random generators and optimal coefficients in a very basic situation. 
We refer for more information to the survey of Korobov \cite{Korobov}, the work of Niederreiter \cite{niederreiter1977pseudo,niederreiter1978quasi} or the book of Konyagin and Shparlinski \cite[Chapter $12$]{igorkonya} and keep their notations. 
For any prime $p$ and integer $1\leq \lambda \leq p-1$ we define
$$\sigma(\lambda,p)
:=\sum_{0<\left\|\mathbf{h}\right\|_{\infty}\leq p-1} \frac{\delta_p(h_1+h_2 \lambda)}{r(\mathbf{h})}$$ 
where $\delta_p(a)= 1$ if $a = 0 \bmod p$ and $\delta_p(a)=0$ otherwise.

For any $\lambda$, we define 
$$\rho(\lambda,p)
=\min_{\mathbf{h} \neq 0} r(\mathbf{h})$$ 
where the min is taken over all non trivial solutions $\mathbf{h}=(h_1,h_2)$ of the congruence 
$$h_1+h_2\lambda = 0 \bmod p.$$ 
These two quantities are relatively close to each other:

\begin{lemma}\label{niedpseudo}
\cite[Theorem $3.8$]{niederreiter1977pseudo}.
There exists $C> 0$ such that, for any prime $p \geq 3$, and $\lambda\in \{1,\dots,p-1\}$ we have 
\begin{equation}\label{niederreiter}
\frac{1}{\rho(\lambda,p)}\leq \sigma(\lambda,p) \leq C \frac{(\log p)^{2}}{\rho(\lambda,p)}.
\end{equation}
\end{lemma} 

In some cases which are of interest for our problem, we can control from below $\rho(\lambda,p)$:

\begin{lemma}\label{minvalue} 
Let $\lambda$ be an element order $k \geq 3$ in the multiplicative group $\mathbb{F}_{p}^*$.
Then 
$$\rho(\lambda,p) \geq p^{1/\phi(k)}/\sqrt{8},$$
where $\phi$ denotes as usual the Euler's totient function.
\end{lemma}

\begin{proof} 
Let 
$$\Phi_k(X)
=\sum_{0\leq l\leq\phi (k)} a_lX^l
=\prod_{1\leq l\leq k\atop\gcd (k,l)=1} (X-\zeta_k^l)$$ 
denote the $k$-th cyclotomic polynomial. 
Set 
$$\Vert\Phi_k(X)\Vert_2
=\left (\sum_{0\leq l\leq\phi (k)}\vert a_l\vert^2\right )^{1/2}
=\left (\frac{1}{2\pi}\int_0^{2\pi}\vert\Phi_k(e^{it})\vert^2dt\right )^{1/2}
\leq 2^{\phi (k)}.$$ 
We clearly have $\Phi_k(\lambda)= 0 \bmod p$. 
For $\mathbf{h}=(h_1,h_2) \neq 0$ we define 
$P(X)=h_1 + h_2X.$ 
Assume that $P(\lambda)=0 \bmod p$, then $p$ divides the resultant $R={\rm Res}(P,\Phi_k)$. 
The polynomial $\Phi_k$ being irreducible of degree $\geq 2$, 
we deduce that $R\neq 0$. 
It follows that $\vert R\vert \geq p$. 
Since $R$ is the determinant of the Sylvester matrix of $P(X)$ and $\Phi_k(X)$, 
by Hadamard's inequality we have
$$ \vert R\vert\leq\Vert P(X)\Vert_2^{\deg\Phi_k(X)}\Vert\Phi_k(X)\Vert_2^{\deg P(X)}
\leq\left(h_1^2+h_2^2\right)^{\phi(k)/2}2^{\phi (k)}
\leq\left(\max(\vert h_1\vert,\vert h_2\vert)\right)^{\phi(k)}8^{\phi (k)/2} .$$ 
Hence we have
$$r(\mathbf{h})
\geq\max(\vert h_1\vert,\vert h_2\vert)
\geq \vert R\vert^{1/\phi(k)}/\sqrt{8} \geq p^{1/\phi(k)}/\sqrt{8}.$$ 
All together we obtain the lower bound 
$\rho(\lambda,p) 
\gg p^{1/\phi(k)}/\sqrt{8}.$
\end{proof}

\subsubsection{Reduction to a problem of equidistribution}
Set $\Ha=\ker(\chi)$, the subgroup of $\mathbb{F}_{p}^*$ of order $d$. 
We interpret the condition $\chi(n_1)=\chi(n_2)$ as $n_1n_2^{-1} \in \Ha$. 
We write $\Ha$ as a disjoint union 
$$\Ha=\bigcup_{k\mid d}\Ha_k, 
\hbox{ where }\Ha_k
:=\{\theta \in \Ha, ord(\theta)=k\}.$$

\begin{proposition}\label{reduction}
For any pair $(x,y)$ of $I_{2}$ we define 
$$f_d(x,y)
= \frac{x}{d-1}+\min(x,y).$$ 
We have the following relation
$${\cal A}(p,d)
=\frac{1}{p-1} \sum_{1\leq n_1,n_2 \leq p-1 \atop \chi(n_1)=\chi(n_2)} \min(n_1,n_2) 
=\frac{p}{(p-1)}
\sum_{k\mid d \atop k\neq 1}\sum_{\theta \in \Ha_k} \sum_{ x \bmod p}f_d\left(\frac{x}{p},\frac{x \theta}{p}\right).$$
\end{proposition}
 
\begin{proof}
Changing the order of summation in (\ref{defApchi}) 
and making the change of variables $(n_1,n_2)\mapsto( p-n_1,p-n_2)$, 
we do have 
$$(p-1){\cal A}(p,d) 
=\sum_{1\leq n_1,n_2\leq p-1\atop\chi (n_1)=\chi(n_2)}\left (p-\max(n_1,n_2)\right )
=\sum_{1\leq n_1,n_2\leq p-1\atop\chi (n_1)=\chi(n_2)}\min(n_1,n_2).$$
Now we have 
$$\sum_{1\leq n_1,n_2 \leq p-1 \atop \chi(n_1)=\chi(n_2)} \min(n_1,n_2)
=\sum_{x\bmod p}\left(x+\sum_{\theta \in \Ha \atop \theta \neq 1} \min(x,\theta x)\right).$$
We remark that if $\theta \neq 1$, we have 
$$\min(x,\theta x) 
= pf_d\left(\frac{x}{p},\frac{x \theta}{p}\right)-\frac{x}{d-1}.$$
Using the decomposition $\Ha=\bigcup_{k\mid d}\Ha_k$ and summing over $\Ha$, the proposition follows. 
\end{proof}
 
\begin{remark}
The reader might wonder why we did not express directly the sum ${\cal A}(p,d)$ using the more natural function on $I_d$ 
given by $g(x_1,\dots,x_d)=\sum_{i=1}^{d} \min(x_1,x_i)$ 
evaluated at the points $\left(\frac{x}{p},\frac{x\lambda}{p},\dots,\frac{x\lambda^{d-1}}{p}\right)$,
where $\lambda$ generates $\Ha$. 
This comes from the fact that these points are not equidistributed in $I_d$ 
because they lie in the hyperplane of equation $x_1+\dots+x_d=0$. 
\end{remark}

\subsubsection{Proof of Theorem \ref{asympdpetit}}
We introduce the set of points in $I_{2}$: 
$$S_{\theta}=\left\{\left(\frac{x}{p},\frac{x \theta}{p}\right), x \bmod p\right\}$$ 
for any $\theta \in \Ha \backslash \{1\}$. 
By Theorem \ref{koksma} we have for any $\theta$
$$\left\vert
\frac{1}{p} \sum_{ x \bmod p} f_d\left(\frac{x}{p},\frac{x \theta}{p}\right) 
-\int_{I_{2}} f_d(u,v)dudv
\right\vert 
\leq V(f_d)D(S_{\theta}).$$
It is easy to compute the integral and obtain 
$$\int_{I_{2}} f_d(u,v)dudv 
=\frac{1}{2(d-1)}+\frac{1}{3}.$$ 
Applying Proposition \ref{reduction} and simplifying, we obtain the equation
\begin{equation}\label{errorterm} 
{\cal A}(p,d)
=\frac{2d+1}{6}p+ O\left(ET \right)
\end{equation}
where the error term is 
\begin{equation}\label{error}
ET
:=pV(f_d)\left(\sum_{k\mid d \atop k\neq 1}\sum_{\theta \in \Ha_k}D(S_{\theta})\right).
\end{equation} 
The readers can easily convince themselves that $V(f_d) \ll 1$ independently of $d$ 
(for instance look at the variation of $f_d$ over the rectangle $R:=\left[x_1,x_2\right] \times \left[x_2,y_2\right]$, 
namely $v_R(f):=f_d(x_2,y_2)-f_d(x_1,y_2)-f_d(x_2,y_1)+f_d(x_1,y_1)$. 
The Vitali variation can then be obtained by summing $v_R(f)$ over a partition of $I_2$ 
and taking the supremum over all possible partitions) 
\footnote{The Hardy-Krause variation is then obtained as a sum of the Vitali variations of $f_d$, $f_d(x,1)$ and $f_d(1,y)$.}.
Hence to finish the proof, we need to bound the sum of discrepancies. 
Applying Theorem \ref{erdosturan} with $H=p-1$ we obtain
$$D(S_{\theta}) 
\leq\left( \frac{3}{2}\right)^{2} 
\left(
\frac{2}{p} 
+\sum_{0<\left\|\mathbf{h}\right\|_{\infty}\leq p-1} 
\frac{1}{r(\mathbf{h})}
\left\vert \frac{1}{p}\sum_{x=1}^{p}e\left(\frac{h_1x+h_2 x\theta}{p}\right)\right\vert
\right).$$ 
Using the orthogonality relations
$$\sum_{b \bmod p} e(bn/p) 
= \left\{\begin{array}{ll}
p,&\quad\text{if $n\equiv 0 \pmod p$,}\\
0,&\quad\text{if $n\not\equiv 0 \pmod p$,}
\end{array}
\right.$$
we can bound the sum over $\mathbf{h}$ by
$$\sigma(\theta,p)
:=\sum_{0<\left\|\mathbf{h}\right\|_{\infty}\leq p-1} \frac{\delta_p(h_1+h_2 \theta)}{r(\mathbf{h})}$$ 
using the notations of subsection \ref{pseudo}. 
For $\theta \in \Ha_k$, we apply consecutively Lemma \ref{niedpseudo} and Lemma \ref{minvalue} to obtain
$$\sigma(\theta,p)
\leq C(\log p)^{2}/p^{1/\phi(k)}$$ 
for an absolute constant $C$. 
Hence recalling that $\gamma(d)=\max_{k \mid d}\phi(k)$ and summing over $k$, 
we arrive at
$$ET
=pV(f_d)\left(\sum_{k\mid d \atop k\neq 1}\sum_{\theta \in \Ha_k}D(S_{\theta})\right)
\ll d(\log p)^{2}p^{1-1/\gamma(d)}.$$ 
This concludes the proof of Theorem \ref{asympdpetit}, in view of Equation \eqref{errorterm}.

\subsection{Asymptotic for large $d$'s}
For a given non-principal Dirichlet character $\chi \mod{p}$, where $p$ is a prime, let
$$M(\chi) 
:= \max_{1 \le x \le p} \left| \sum_{n \le x} \chi(n) \right|$$
 and its renormalization
$$m(\chi) 
= \frac{M(\chi)}{e^\gamma\sqrt{p}/\pi}.$$
The P\'{o}lya--Vinogradov Theorem states that
\begin{equation}\label{Polya}
m(\chi) \ll \log p \end{equation}
for all non-principal characters $\chi\mod{p}$. 
Apart from some improvements on the implicit constant, this remains the state-of-the-art for the general non-principal character. 
However, for most of the characters $M(\chi)$ is much smaller and we can study how often $M(\chi)$ is large. 
The best result in this direction was obtained in \cite{bober2018frequency}:

\begin{theorem}\label{frequency} 
Let $\eta=e^{-\gamma}\log2$. If $1\le \tau\le \log_2  p-M$ for some $M\ge4$, then
$$ \Phi_p(\tau) 
:= \frac{1}{p-1} \# \left\{ \chi\mod p: m(\chi)> \tau \right\}
\leq \exp\left\{ - \frac{e^{\tau -2 - \eta }}{\tau} (1+O((\log\tau) /\tau))\right\}.$$
\end{theorem}

We are now in a position to prove Theorem \ref{conjecture} for large $d$'s.

\noindent\frame{\vbox{
\begin{theorem}\label{asymptoticlarge}
 Assume that $d\rightarrow +\infty$ and $\log d= o(\log p/\log_2 p)$. Then we have the following asymptotic formula
$${\cal A}(p,d)= dp/3 + O\left(p (\log_2 d)^2 \right).$$  For any $d$ (in particular for larger $d$'s), the following holds
 $${\cal A}(p,d)= dp/3 + O\left(p (\log p)^2 \right).$$
\end{theorem}
}}

\begin{remark} The condition $\log d= o(\log p/\log_2 p)$ could be made more explicit by specifying constants in the proof below.
Notice also that whereas Theorem \ref{conjecture} follows from Theorems  \ref{asympdpetit} and \ref{asymptoticlarge}, it does not follow from Theorem \ref{asympdpetit} and \eqref{Elmaremark}.
\end{remark}

\begin{proof} The second part of the Theorem follows directly from \eqref{MA} and the inequality $\vert L(1,\chi) \ll \log p$. This could also be proved following our argument below and using only P\'{o}lya--Vinogradov inequality. Let us now focus on the case $\log d= o(\log p/\log_2 p)$.
The condition $\chi(n_1)=\chi(n_2)$ is equivalent to $n_1n_2^{-1}$ lying in the kernel of $\chi$, which is a subgroup of order $d$ of the multiplicative cyclic group $\mathbb{F}_{p}^*$. We apply the
 orthogonality of characters in the subgroup $<\chi >$ of order $m$ generated by $\chi\in \cX_p$ and
 rewrite the sum ${\cal A}(p,d)$ defined in (\ref{defApchi}) as
$${\cal A}(p,d)
=\frac{1}{(p-1)}\sum_{N=1}^{p-1} \frac{1}{m}\sum_{\Psi \in \cX_p \atop \Psi^m= \chi_0}  \sum_{1 \leq n_1,n_2 \leq N} \Psi(n_1n_2^{-1}).$$
Separating the contribution of the trivial character from the others, this leads us to the equation
$${\cal A}(p,d)
= \frac{d}{(p-1)^2}\sum_{N=1}^{p-1} N^2 
+ \frac{1}{(p-1)}
\sum_{N=1}^{p-1}\frac{1}{m}\sum_{\Psi \in \cX_p^* \atop \Psi^m= \chi_0} \left\vert \sum_{1 \leq n \leq N}\Psi(n)\right\vert^2.$$ 
We have trivially
$$\frac{d}{(p-1)^2}\sum_{N=1}^{p-1} N^2 
=\frac{dp}{3}
+\frac{dp}{6(p-1)}
= dp/3 + O(d).$$ 
Therefore we are left to bound the contribution of non-trivial characters and 
\begin{equation}\label{erreur} {\cal A}(p,d)
=dp/3 + O\left(d\frac{R}{(p-1)^2}\right)\end{equation}
where 
$$R
:=\sum_{N=1}^{p-1}\sum_{\Psi \in \cX_p^* \atop \Psi^m= \chi_0} \left\vert \sum_{1 \leq n \leq N}\Psi(n)\right\vert^2. $$ Let us set the parameter $\tau=\min\{C(\log_2 d),\log_2 p-M\}$ where $M$ is the constant appearing in Theorem \ref{frequency} and $C$ is some large constant which will be specified later. We introduce the following set of characters $$\cX_{p,0}^{\tau}= \left\{ \Psi\in \cX_p^*:  m(\Psi)\leq \tau \right\}$$  and further define for every integer $1\leq j\leq J$

$$\cX_{p,j}^{\tau}:= \left\{ \Psi\in \cX_p^*: 2^{j-1}\tau < m(\Psi)\leq 2^j\tau \right\}$$ where $J$ is chosen in order to allow an application of Theorem \ref{frequency}. Precisely, we choose $J$ such that 
$$\tau 2^{J} \leq \log_2 p - M < \tau 2^{J+1}.$$  We now split the characters appearing in the summation in $R$ as follows
$$\cX_p^* =\left(\bigcup_{j=0}^{J} \cX_{p,j}^{\tau}\right)\bigcup \left\{ \Psi\in \cX_p^*: m(\Psi)> 2^J \tau \right\}.$$ Notice that if $\tau=\log_2 p-M$ then $J=0$ and we only split the summation depending on whether $ m(\Psi)\leq \log_2 p - M$ or not. Remark that here are at most $m$ characters $\Psi \in \cX_{p,0}^{\tau}$ appearing in the sum. Hence, it follows from Theorem \ref{frequency} and the inequality \eqref{Polya} that
\begin{eqnarray}\label{splitting}
R & \ll & \sum_{N=1}^{p-1}\left(mp\tau^2 + p^2\sum_{j=1}^{J}\tau^2 2^{2j} \Phi_p(\tau 2^{j-1}) + p^2(\log p)^2 \Phi_p(\tau 2^J)\right) \nonumber \\ 
&\ll & \sum_{N=1}^{p-1}\left(mp\tau^2 + p^2\sum_{j=1}^{J}\tau^2 2^{2j} \exp\left\{-\frac{e^{ \tau 2^{j-1}}}{100\tau 2^{j}}\right\} + p^2(\log p)^ 2 \exp\left\{ -\frac{e^{\tau 2^J}}{100\tau 2^J}\right\}\right). \end{eqnarray} The summation over $j$ in the right hand side of \eqref{splitting} is clearly dominated by its first term. Thus we obtain after summing over $N$ and recalling our choice of $J$:
\begin{eqnarray}\label{erreurfinale}
R \ll \frac{p^3}{d}\tau^2 +p^3\tau^2 \exp\left\{-c_1\frac{e^{ \tau}}{\tau }\right\}+ p^3(\log p)^2 e^{-c_2 \log p/ \log \log p} 
\end{eqnarray} for some absolute constants $c_1,c_2 >0$. We insert \eqref{erreurfinale} in \eqref{erreur} and choose $C$ large enough in the definition of $\tau$ to ensure that the second and third term in the right hand side of \eqref{erreurfinale} have negligible contribution. This is indeed possible due to the restriction on the size of $d$ and concludes the proof.
\end{proof}

\section{Proof of Theorem \ref{mainth1}}

The first part of Theorem \ref{mainth1} follows from Theorem \ref{asympdpetit} and \eqref{MA}. The second part follows from Theorem \ref{asymptoticlarge} and \eqref{MA}.

\section{Concluding remarks}

  We solved Elma's question about the asymptotic behavior of the character sums ${\cal A}(p,d)$  regardless of the size of $d$. As already noticed above, for $d$ large, this is not precise enough to deduce an asymptotic formula for the mean-square value $M(p,m)$. To conclude, let us say that the upper bound \eqref{boundMpmlarge} could be obtained by working directly with $L(1,\chi)$ following our method of proof of Theorem \ref{asymptoticlarge}. This requires results about the distribution of $L(1,\chi)$  as the ones obtained by Granville and Soundararajan \cite{GS1,GS2} instead of Theorem \ref{frequency}.

\section*{Funding}
This work was supported (for M. M) by the Austrian Science Fund
(FWF) [P-33043].

 \section*{Acknowledgements}
The second author would like to thank Igor Shparlinski for sketching a refinement of our argument in the proof of Theorem \ref{asymptoticlarge} leading to a better result.

{\small
\bibliography{central}

}
\end{document}